\documentclass[a4paper,centertags,osxzneside,12pt]{amsart}
\usepackage{amssymb}
\usepackage{mathrsfs}
\usepackage{appendix}
\usepackage{amssymb}
\usepackage{fancyhdr}
\usepackage{charter}
\usepackage{typearea}
\usepackage{pdfsync}
\usepackage{color}
\usepackage[colorlinks,linkcolor=blue,anchorcolor=blue,citecolor=green]{hyperref}
\usepackage{mathrsfs}
\usepackage[a4paper,top=3cm,bottom=3cm,left=3cm,right=3cm]{geometry}

\usepackage{enumitem}
\usepackage{color}
\setlist[1]{itemsep=5pt}
\newcommand{\comment}[1]{}

\makeatletter
      \def\@setcopyright{}
      \def\serieslogo@{}
      \makeatother

\newtheorem{theorem}{Theorem}[section]

\newtheorem{corollary}[theorem]{Corollary}

\newtheorem{remark}[theorem]{Remark}

\numberwithin{equation}{section}

\begin{document}
\title{CSC Bergman metrics are Einstein on  pseudoconvex domains}

\author{Xiaoshan Li}
\address{School of Mathematics and Statistics, Wuhan University, Wuhan, Hubei 430072, China.}
\email{xiaoshanli@whu.edu.cn}
\thanks{Xiaoshan Li is supported in part by NSFC (12361131577, 12271411)}

\author{Taishun Liu}
\address{Department of Mathematics, Huzhou normal University, Huzhou 313000, China.}
\email{lts@ustc.edu.cn}
\date{April 2026}

\maketitle
\begin{center}
\vspace{-8pt} \small \itshape
  Dedicated to Professor Xiaojun Huang, with  admiration for  his mathematical contributions
\vspace{-5pt} 
\end{center}

\begin{abstract}
Let $\Omega\subset\mathbb C^n$($n\geq 3$) be a  pseudoconvex domain
with a smooth strongly pseudoconvex boundary point.  We
prove that if the Bergman metric of $\Omega$ has constant scalar
curvature, then it is necessarily K\"ahler--Einstein. 
\end{abstract}

\section{csc Bergman metric}
Let \(\Omega\subset\mathbb C^n\) be a possibly unbounded pseudoconvex domain whose boundary \(\partial\Omega\) contains a smooth strongly pseudoconvex point \(p\). Let \(\{\phi_j\}\) be an orthonormal basis of the Bergman space \(A^2(\Omega)\), the subspace of \(L^2(\Omega)\) consisting of square integrable holomorphic functions on \(\Omega\). Write
$K_\Omega(z,z)=\sum_j|\phi_j(z)|^2$
for the Bergman kernel function on \(\Omega\), and let
\[
g_\Omega=\sum_{i,j}g_{i\overline j}\,dz_i\otimes d\overline z_j,\qquad
g_{i\overline j}=\frac{\partial^2}{\partial z_i\,\partial\overline z_j}\log K_\Omega(z,z)
\]
be the Bergman metric of \(\Omega\). This metric is well defined on a maximal open subset \(\Omega^*\subset\Omega\) that contains a one-sided neighborhood of \(p\) (see \cite{HJL25}). By a result of Huang--Li \cite{HLi26}, \(\Omega\setminus\Omega^*\) is contained in a complex hypersurface (possibly singular) of \(\Omega\).

Set \(G_\Omega=(g_{i\overline j})_{n\times n}\) and on $\Omega^*$ we define the Bergman invariant function
\[
J_\Omega=\frac{\det G_\Omega}{K_\Omega}.
\]
For the Bergman metric \(g_{\Omega}\), the Ricci curvature tensor is given by
\[
\mathrm{Ric}_{i\overline j}
=
-\frac{\partial^2}{\partial z^i\,\partial\overline z^j}
\log\det(G_{\Omega}),
\]
and the scalar curvature is the trace of the Ricci tensor with respect to the metric:
\[
S_\Omega
=\sum_{i, j}
g^{i\overline j}\mathrm{Ric}_{i\overline j}
=\sum_{i, j}
-g^{i\overline j}\frac{\partial^2}{\partial z^i\,\partial\overline z^j}
\log\det(G_\Omega).
\]
It is known that near a strongly pseudoconvex boundary point $p$ \cite{KYu96}, the scalar curvature \(S_\Omega\) and the Bergman invariant function have limits:
\[
\lim_{z\rightarrow p}S_{\Omega}(z)=-n,\quad \lim_{z\rightarrow p}J_{\Omega}(z)=c_n=\frac{(n+1)^n\pi^n}{n!}.
\]

Starting from the identity
\[
\log J_\Omega=\log\det G_\Omega-\log K_\Omega,
\]
and applying \(\partial_{z_k}\partial_{\overline z_j}\), then contracting with \(g^{k\overline j}\), we obtain that \(S_\Omega\) is constant if and only if \(\log J_\Omega\) is harmonic with respect to the Bergman metric. This was first derived by Sha (see \cite{S26}):
\[
\Delta_{g_\Omega}\log J_\Omega 
=
\sum_{j,k}g^{k\overline j}
\frac{\partial^2}{\partial z^k\,\partial\overline z^j}
\log J_\Omega
=
0
\qquad \text{on }\Omega^\ast.
\]

When the Bergman metric \(g_\Omega\) has constant scalar curvature on \(\Omega^*\), we say that \(\Omega\) admits a constant scalar curvature (csc) Bergman metric. Since this condition can be expressed as a real analytic equation on \(\Omega\), it follows from the uniqueness of real analytic functions that \(\Omega\) has a csc Bergman metric if and only if \(S_\Omega\) is constant on some nonempty open subset of \(\Omega^\ast\).

In this paper, we present a proof of the following result:
\begin{theorem} \label{main}Let $\Omega
\subset {\mathbb C}^n$ with $n\ge 3$ be a pseudoconvex domain with $p\in \partial\Omega$ a  smooth strongly pseudoconvex  boundary point of $\Omega$. If the Bergman metric of $\Omega$ has constant scalar curvature, then it is K\"ahler--Einstein.
\end{theorem}
In particular, we have the following
\begin{corollary}
    Let $\Omega\Subset\mathbb C^n$ with $n\geq 3$ be a smooth bounded pseudoconvex domain. If the Bergman metric of $\Omega$ has constant scalar curvature, then it  is K\"ahler--Einstein.
\end{corollary}

\begin{remark}
    Our results are stated for $n\geq 3$. However, in the two dimensional case, if we assume in addition that some boundary point $p\in\partial\Omega$ is spherical, then under the same assumptions as in Theorem \ref{main} the same argument yields  the Bergman metric of $\Omega$ is K\"ahler--Einstein.
\end{remark}

The proof of Theorem \ref{main} proceeds in two steps. First, building up on the sphericity result in \cite{HL26}, we show that
\(J_\Omega\) vanishes to infinite order along \(\partial\Omega\) near \(p\).
Together with Bell's result \cite{Be87}, which yields the real analyticity of
\(J_\Omega\) when \(\partial\Omega\) is real analytic, this already implies
that the Bergman metric of \(\Omega\) is K\"ahler--Einstein in the
real-analytic category. This part is essentially taken from  an earlier
version of~\cite{HL26}.  (The later arXiv version \cite{HL26} contains the vanishing of second order as it was written as an appendix for \cite{HHL26} and that is all needed for application in \cite{HHL26}.)
The second step is to adapt the relevant boundary unique continuation result in order
to establish the K\"ahler--Einstein property in the smooth category as well. In particular, we adapt the boundary unique continuation method of Biquard and Herzlich \cite{BH14}. Related boundary unique continuation problems have recently attracted considerable attention. We refer to Berhanu’s survey \cite{Ber21} and his resolution of the celebrated Baouendi–Rothschild conjecture for elliptic PDEs with real-analytic coefficients \cite{Ber25}.

\bigskip

{\bf Acknowledgment}: 
The first author expresses his sincere gratitude to  Professor Xiaojun Huang, for his guidance, support, and
encouragement over the years. This paper is particularly indebted
to him, as it is a natural continuation of our previous note
\cite{HL26}. An earlier version of that note already contained
a solution in the real-analytic category, but Professor Huang
urged us to omit that case and encouraged the authors independently  to pursue the
problem in its full generality. He also had extensive discussions
with Professor S. Chanillo on Carleman estimates and boundary
unique continuation for subelliptic differential operators
and ACH Laplacians, and shared the resulting references and
insights with us in May 2026. We  thank Professor S. Chanillo
for his contributions through these discussions.

\section{vanishing to infinite order of Bergman invariant function near a stronlgy pseudoconvex boundary point}

Let $\Omega\subset\mathbb C^n$ be a pseudoconvex domain with $p\in\partial\Omega$ a strongly pseudoconvex boundary point. Let \(\rho\) be a defining function for
\(\Omega\) near \(p\in\partial\Omega\). For
\(f\in C^\infty(\overline{\Omega}\cap U)\), we write $f=O(\rho^k)$ if there exist a $f_k\in C^\infty(\overline\Omega\cap U)$ such that $f=\rho^k f_k$, and 
$
f=O(\rho^\infty)
$
near \(p\) if $f=O(\rho^k)$ for each $k\in\mathbb N$.
Equivalently, \(f\) vanishes to
infinite order along the boundary patch under consideration.

In this section, we prove that \(J_\Omega-c_n\) vanishes to infinite
order along the boundary near $p$,
under the assumption that the Bergman metric has constant scalar
curvature.

The proof of this property is quite similar to that of the corresponding  known result (see, e.g.,  \cite{HL23}) where the CSC condition was replaced by the slightly stronger  K\"ahler--Einstein condition. The new observation here is to  employ  the Martin expansion formula of $J_\Omega$ \cite{Ma21} besides the formula of Christoffers for the Bergman expansion \cite{Ch81}.  The asymptotic expansion formulas of Engliš \cite{Eng08} also play an important role.

\begin{theorem}\label{9-16-lem1}
    Let $\Omega\subset {\mathbb C}^n$ with $n\ge 3$ be a pseudoconvex domain with $p\in \partial\Omega$ a  smooth strongly pseudoconvex  boundary point. If the Bergman metric of $\Omega$ has constant scalar curvature, then $$J_{\Omega}(z)=c_n+O(\rho^\infty),~~z\approx p$$ where $\rho$ is a defining function of $\Omega$ near $p$.
\end{theorem}

A particular consequence is when $\partial \Omega$ is real analytic near $p$, by a result of Bell \cite{Be87} and its slight generalization, $J_{\Omega}$  can be extended to a  real analytic function in a neighborhood of $p$ in $\mathbb C^n$ and thus $J_{\Omega}$ is a constant and thus the Bergman metric is K\"ahler--Einstein.

\begin{proof}
First, under the csc assumption of the Bergman metric, it follows from a result of Huang-Li \cite{HL26} that $p$ is a spherical boundary point. 

By sphericity of $p$, there is a neighborhood $U$ of $p$ in $\mathbb C^n$ and a neighborhood $V$ of $q$ and a biholomorphic map $F: \Omega\cap U\rightarrow \mathbb B^n\cap V$ which extends smoothly to the boundary and maps $\partial\Omega\cap U$ onto $V\cap\partial\mathbb B^n$. Define
    $$r_0=(1-|F(z)|^2)|\det F'(z)|^{-\frac{2}{n+1}}.$$
    Then $r_0$ is a smooth defining function near $p$.
    For any real-valued  $C^2$-smooth function $u$, we define the Fefferman-Monge-Ampere operator:
\begin{equation}\label{MA}
\mathcal J(u)=(-1)^n\det\left[\begin{array}{cc}
     u & u_{\overline\beta}  \\
     u_{\alpha}&  u_{\alpha\overline\beta}
\end{array}\right].
\end{equation}
    The transformation law for Fefferman--Monge--Ampere operator implies that
    $$\mathcal J(r_0)=1.$$
By localization of Bergman kernels, we have
\begin{equation}\label{localization}
\begin{split}
&K_\Omega(z, w)-K_{\Omega\cap U}(z, w)\in C^\infty((U\cap\overline\Omega)\times(U\cap\overline\Omega));\\
&K_{ \mathbb B^n}(\xi, \eta)-K_{\mathbb B^n\cap V}(\xi, \eta)\in C^\infty((V\cap\overline{\mathbb{B}^n})\times (V\cap\overline{\mathbb B^n})).
    \end{split}
\end{equation}
From the transformation law of Bergman kernels, we have 
\begin{equation}\label{transformation of Bergman}
K_{\Omega\cap U}(z, w)=K_{\mathbb B^n\cap V}(F(z), F(w))\det F'(z)\overline{\det F'(w)}.
\end{equation}
Then near $p$ we define 
$$K_0(z, w):=K_{\mathbb B^n}(F(z), F(w))\det F'(z)\overline{\det F'(w)}.$$
It follows from (\ref{localization}) and (\ref{localization}) that 
$$K_{\Omega}(z, w)-K_0(z, w)\in C^\infty((U\cap\overline\Omega)\times(U\cap\overline\Omega)).$$
Consequently, on the diagonal,
$$K_{\Omega}(z, z)=K_0(z, z)+E(z), E(z)\in C^\infty(U\cap\overline\Omega).$$
From the explicit formal of the Bergman kernel of $\mathbb B^n$, we have
$$K_0(z, z)=\frac{n!}{\pi^n}\frac{1}{r_0^{n+1}}$$
and thus,
\begin{equation}\label{9-13-a1}
    K_{\Omega}(z, z)=\frac{1}{C_nr_0^{n+1}}+E(z),
\end{equation}
where $C_n=\frac{\pi^n}{n!}$.

We introduce the Bergman defining function $$u_\Omega=(C_n K_{\Omega})^{-\frac{1}{n+1}}.$$
Since $E(z)$ is smooth up to the boundary near $p$, then $u_{\Omega}$ is another smooth defining function for $\Omega$ near $p$. Moreover,
\begin{equation}\label{Bergman defining function}
\begin{split}
    u_{\Omega}&=r_0(1+C_n E(z)r_0^{n+1})^{-\frac{1}{n+1}}=r_0-\frac{C_n}{n+1}Er_0^{n+2}+O(r_0^{n+3})\\
    &=r_0+ar_0^{n+2}+O(r_0^{n+3})
    \end{split}
\end{equation}
where $a=-\frac{C_n}{n+1}E$.

We now use the standard  identity for Fefferman--Monge--Ampere operator \cite[p. 191]{HK97}. If $r$ is a smooth strongly plurisubharmonic defining function and $$u=r+\varphi r^m+O(r^{m+1}),$$
then
\begin{equation}\label{identity for J}
    \mathcal J(u)=\mathcal J(r)+m(m-n-2)\varphi r^{m-1}+O(r^m).
\end{equation}

Substituting (\ref{Bergman defining function}) to (\ref{identity for J}), we have
\begin{equation}
    \mathcal J(u_{\Omega})=\mathcal J(r_0)+O(r_0^{n+2})=1+O(r_0^{n+2}).
\end{equation}
On the other hand, we have the Fefferman formula (see \cite{Ma21}) 
\begin{equation}\label{Fefferman formula}
    J_{\Omega}=\frac{(n+1)^n\pi^n}{n!} \mathcal J(u_\Omega)
\end{equation}
It follows from (\ref{Fefferman formula}) that 
\begin{equation}\label{9-12-a1}
    J_{\Omega}=c_n+O(r_0^{n+2}).
\end{equation}
Hence, 
$$J_{\Omega}=c_n+O(\rho^{n+2}).$$
It follows that 
$$\log J_{\Omega}=\log c_n+O(\rho^{n+2}).$$

We write $\log J_\Omega$ in the following asymptotic expansion  near $p$, 
\begin{equation}\label{expansion of log J}
    \log J_{\Omega}\sim \log c_n+a_{n+2}\rho^{n+2}+a_{n+3}\rho^{n+3}+\cdots.
\end{equation}

That is, for any $m\in\mathbb N$, the difference $$\log J_\Omega(z)-\log c_n-\sum_{j=n+2}^ma_j\rho^j=O(\rho^{m+1}).$$

Since 
 $\Delta_{g_\Omega}\log J_{\Omega}=0$, that is, 
 $$\sum_{i, j}g^{\overline j i}\frac{\partial^2}{\partial z_i\partial\overline z_j}\log J_{\Omega}=0.$$
By direct calculation, for $m\geq n+2$
\begin{equation}\label{9-16-a1}
\begin{split}
\frac{\partial^2}{\partial z_i \partial \overline{z}_j} \left( a_m \rho^m \right)
&= \rho^m \frac{\partial^2 a_m}{\partial z_i \partial \overline{z}_j}  + m \rho^{m-1} \left(
\frac{\partial a_m}{\partial z_i} \frac{\partial \rho}{\partial \overline{z}_j}
+ \frac{\partial a_m}{\partial \overline{z}_j} \frac{\partial \rho}{\partial z_i}
\right) \\
& + a_m m \rho^{m-1} \frac{\partial^2 \rho}{\partial z_i \partial \overline{z}_j} + a_m m(m-1) \rho^{m-2} \frac{\partial \rho}{\partial z_i} \frac{\partial \rho}{\partial \overline{z}_j}.
\end{split}
\end{equation}
From \cite[Section 3]{Eng08} we have
\begin{equation}\label{5-23-a2}
\begin{split}
    &\rho^{-1}g^{\overline j i}=\rho^{-1}[\log \rho]^{\overline j l}H^i_l, H^i_l\in C^n(U\cap\overline G), H^i_l|_{\partial G}=-\frac{1}{n+1}\delta^i_l.\\
    &\frac{1}{\rho^2}[\log\rho]^{\overline j i}\rho_i, ~~\frac{1}{\rho^2}[\log\rho]^{\overline j i}\rho_{\overline j}\in C^\infty(\overline G),~[\log \rho]^{\overline j i}\rho_{\overline j}\rho_i=\rho [\log \rho]^{\overline j i}\rho_{i\overline j}-n\rho^2.\\
    &\lim_{z\rightarrow q} \frac{1}{\rho^2}[\log \rho]^{\overline j i}\rho_{\overline j}\rho_i=-1, \quad q\in\partial\Omega, q\approx p.
    \end{split}
\end{equation}
By (\ref{9-16-a1}) and (\ref{5-23-a2}), for $m\geq n+2$, we have for $q\in\partial\Omega$ and $q\approx p$ that
\begin{equation}\label{5-27-a1}
\begin{split}
&\lim_{z\rightarrow q}\rho^{-m} g^{\overline j i}
  \frac{\partial^2}{\partial z_i \partial \overline z_j} (a_m \rho^m) \\
={}&
\begin{aligned}[t]
&\lim_{z\rightarrow q}\Bigg[
g^{\overline j i} \frac{\partial^2 a_m}{\partial z_i \partial \overline z_j}
+ \frac{m}{\rho} g^{\overline j i}
\left(
\frac{\partial a_m}{\partial \overline z_j} \frac{\partial \rho}{\partial z_i}
+ \frac{\partial a_m}{\partial z_i} \frac{\partial \rho}{\partial \overline z_j}
\right) \\
&\qquad
+ \frac{m(m-1)a_m}{\rho^2} g^{\overline j i}
\frac{\partial \rho}{\partial z_i} \frac{\partial \rho}{\partial \overline z_j}
+ \frac{m a_m}{\rho} g^{\overline j i}
\frac{\partial^2 \rho}{\partial z_i \partial \overline z_j}
\Bigg]
\end{aligned}\\
={}& m(n-m)a_m(q).
\end{split}
\end{equation}
Hence, the condition $\Delta_{g_\Omega}(\log J_{\Omega})=0$ implies that $a_m(q)=0,~q\approx p$ for all $m\geq n+2$. Thus, $\log J_{\Omega}=\log c_n+O(\rho^\infty)$. That is, $J_\Omega=c_n+O(\rho^\infty)$ near $p$.
\end{proof}

\section{Boundary unique continuation near spherical boundary points}

In this section, we adapt the local half--ball method of Biquard and
Herzlich \cite{BH14} to the Bergman Laplacian. We use their
complex hyperbolic half--ball coordinates, weighted Hölder spaces,
special defining function, and Carleman estimate. Our task is to
verify that the Bergman metric near a spherical boundary point
satisfies the geometric and weighted hypotheses of their argument.
All equation and page references to \cite{BH14} below refer to its
arXiv version. We will prove the following theorem.
\begin{theorem}\label{thm:local-unique-continuation}
Let \(\Omega\subset\mathbb C^n\) be a pseudoconvex domain, and let
\(p\in\partial\Omega\) be a spherical boundary point.  Suppose that
\(f\) is defined in a one-sided neighborhood of \(p\), smooth up
to the boundary, and satisfies
\[
 \Delta_{g_\Omega}f=0,
 \qquad
 f=O(\rho^\infty)
\]
along a boundary neighborhood of \(p\).  Then \(f\equiv0\) near \(p\).
\end{theorem}

\subsection{The local half--ball and the Bergman metric}

After shrinking the spherical boundary patch, local sphericity gives
a one-sided biholomorphism, smooth up to the boundary,
\[
 F:U\cap\Omega\longrightarrow V\cap\mathbb B^n,
 \qquad
 F(p)=q\in\partial\mathbb B^n.
\]
We use the complex-hyperbolic half--ball construction of
\cite[Sections~1--3]{BH14}, transported from the Siegel realization
to the ball model.  Thus we may choose a model half--ball
\(H_{\rm mod}\subset\mathbb B^n\) whose open ideal boundary face
\(D_{\rm mod}\) satisfies
\[
 q\in D_{\rm mod},
 \qquad
 \overline{D_{\rm mod}}
 \Subset V\cap\partial\mathbb B^n.
\]
We denote the Biquard--Herzlich coordinates by
$
 (s_{\rm mod},\tau_{\rm mod},
   \varrho_{\rm mod},y_{\rm mod})
$
and set
\[
 u_{\rm mod}
 :=\operatorname{sech}^{2}(s_{\rm mod}/2),
 \qquad
 v_{\rm mod}
 :=\operatorname{sech}(\varrho_{\rm mod}/2),
\]
as in \cite[Lemma~1.1 and Lemma~3.3]{BH14}.  In particular, the level
hypersurfaces of \(u_{\rm mod}\) are complete with respect to their
induced model metrics.

For sufficiently small \(\varepsilon>0\), let
\[
 H_{\rm mod,\varepsilon}
 :=
 H_{\rm mod}\cap\{0<u_{\rm mod}<\varepsilon\}
 \subset V\cap\mathbb B^n
\]
and define
\begin{equation*}\label{eq:local-half-ball}
 E_{p,\varepsilon}
 :=
 F^{-1}(H_{\rm mod,\varepsilon})
 \subset U\cap\Omega.
\end{equation*}
We use the same symbols
$
 (s,\tau,\varrho,y), u, v
$
for the corresponding functions transported to
\(E_{p,\varepsilon}\) by \(F\).  The weighted Hölder spaces below are
likewise the Biquard--Herzlich spaces
\cite[(29)--(30) and (47)]{BH14}, transported by \(F\).

Define
\begin{equation}\label{eq:model-defining-function}
 r_0
 :=
 (1-|F|^2)|\det F'|^{-2/(n+1)}.
\end{equation}
The pullback of the diagonal Bergman kernel of the ball is
\[
 K_0(z,z)
 =
 \frac{n!}{\pi^n}r_0^{-(n+1)}.
\]
Here \(K_0\) is a model kernel, not the Bergman kernel of
\(U\cap\Omega\).  Localization of the Bergman kernel on the spherical
boundary patch gives
\[
 K_\Omega(z,z)-K_0(z,z)
 \in C^\infty(\overline\Omega\cap U).
\]
Consequently, after shrinking \(U\),
\begin{equation}\label{eq:kernel-comparison}
 K_\Omega
 =
 \frac{n!}{\pi^n}r_0^{-(n+1)}
 \bigl(1+r_0^{n+1}h\bigr),
 \qquad
 h\in C^\infty(\overline\Omega\cap U).
\end{equation}

Let
\[
 g_{\rm ch}
 :=
 \partial\bar\partial\log K_0
 =
 (n+1)\partial\bar\partial(-\log r_0)
\]
be the pullback of the Hermitian Bergman metric of \(\mathbb B^n\).
Put
\[
 g_{\rm ch}^{\mathbb R}
 :=2\operatorname{Re}g_{\rm ch},
 \qquad
 \widehat g_{\rm ch}
 :=\frac{2}{n+1}g_{\rm ch}^{\mathbb R},
 \qquad
 \widehat g
 :=\frac{2}{n+1}g_\Omega^{\mathbb R}.
\]
It follows from \eqref{eq:kernel-comparison} that
\begin{equation}\label{eq:metric-comparison}
 \widehat g
 =
 \widehat g_{\rm ch}+Q,
 \qquad
 Q
 =
 \frac{4}{n+1}\operatorname{Re}\partial\bar\partial
 \log(1+r_0^{n+1}h).
\end{equation}

We estimate \(Q\) in an ACH frame.  On the collar
\(0\leq r_0<\sigma\), let
\[
 M_t:=\{r_0=t\},
 \qquad
 \theta:=\frac{i}{2}(\bar\partial r_0-\partial r_0),
 \qquad
 H_t:=\ker(\theta|_{TM_t}),
\]
and let
\[
 h_t(X,Y):=-d\theta(X,JY),
 \qquad X,Y\in H_t,
\]
be the Levi metric.  Let \(T_t\) be the Reeb vector field of
\(\theta|_{M_t}\), choose local \(h_t\)-orthonormal frames
\(Y_1,\ldots,Y_{2n-2}\) of \(H_t\), and choose a smooth transverse
field \(N\) satisfying \(dr_0(N)=1\).

With
\[
 \mathcal L_{r_0}
 :=
 -2\operatorname{Re}(\partial\bar\partial r_0),
\]
one has
\begin{equation}\label{eq:exact-ch-metric}
 \widehat g_{\rm ch}
 =
 \frac{dr_0\otimes dr_0}{r_0^2}
 +4\frac{\theta\otimes\theta}{r_0^2}
 +2\frac{\mathcal L_{r_0}}{r_0}.
\end{equation}
It follows that
\begin{equation}\label{eq:ACH-frame}
 r_0N,
 \qquad
 \frac{r_0}{2}T_t,
 \qquad
 \sqrt{\frac{r_0}{2}}Y_1,\ldots,
 \sqrt{\frac{r_0}{2}}Y_{2n-2}
\end{equation}
is uniformly equivalent to a
\(\widehat g_{\rm ch}\)-orthonormal frame. This frame is called an ACH frame.

The metric coefficients, inverse metric coefficients, structure
coefficients, and all their ACH derivatives are uniformly bounded in
these local frames.  The Koszul formula therefore gives the same
boundedness for the connection coefficients of
\(\widehat g_{\rm ch}\).  Since the normalized normal and Reeb
derivatives contain a factor \(r_0\), while normalized horizontal
derivatives contain a factor \(\sqrt{r_0}\), ACH differentiation does
not lower the order of vanishing in \(r_0\).  Applying this observation
to \eqref{eq:metric-comparison} gives
\begin{equation}\label{eq:Q-estimate}
 \left|
 (\nabla^{\widehat g_{\rm ch}})^kQ
 \right|_{\widehat g_{\rm ch}}
 =
 O(r_0^{n+1})
 \qquad(k\geq0).
\end{equation}
After decreasing the collar size,
\(\widehat g\) and \(\widehat g_{\rm ch}\) are uniformly
quasi-isometric, and
\begin{equation}\label{eq:connection-comparison}
 \left|
 (\nabla^{\widehat g_{\rm ch}})^k
 \bigl(
   \nabla^{\widehat g}
   -\nabla^{\widehat g_{\rm ch}}
 \bigr)
 \right|_{\widehat g_{\rm ch}}
 =
 O(r_0^{n+1})
 \qquad(k\geq0).
\end{equation}

Throughout this section, the weighted Hölder norms are defined using
the reference metric \(\widehat g_{\rm ch}\).  Let
\(\xi_{\rm BH}\) denote the boundary defining function used in the
single-weight spaces of \cite[(47)]{BH14}.  Since \(r_0\) and
\(\xi_{\rm BH}\) define the same boundary face,
$
 r_0=a\,\xi_{\rm BH},
$
where \(a>0\), and \(a\), \(a^{-1}\), together with all their ACH
derivatives, are uniformly bounded on the closure of the half--ball
end.  Thus the corresponding single-weight spaces are equivalent.

Equation \eqref{eq:Q-estimate}, with one additional derivative used
to control the Hölder seminorm, gives
\[
 Q\in C^{k,\alpha}_{n+1}
 \qquad(k\geq0,\ 0<\alpha<1).
\]
The single to double-weight inclusion
\cite[Lemma~4.1]{BH14} therefore yields
\begin{equation}\label{eq:Q-double-weight}
 Q\in C^{k,\alpha}_{n+1,n+1}
 \qquad(k\geq0,\ 0<\alpha<1).
\end{equation}

The normalized model metric \(\widehat g_{\rm ch}\) is the
complex-hyperbolic metric in the normalization of \cite{BH14}.
Equation \eqref{eq:Q-double-weight} shows that \(\widehat g\) satisfies
the half--ball asymptotic assumptions used in
\cite[Section~8]{BH14}.  We may therefore use the special defining
function constructed in \cite[(68)]{BH14}:
\begin{equation}\label{eq:special-defining-function}
 x=e^{2\phi}u,
 \qquad
 \phi=O(\sqrt{uv}),
 \qquad
 \left|\frac{dx}{x}\right|_{\widehat g}=1.
\end{equation}
In the resulting geodesic coordinates,
\[
 \widehat g
 =
 \frac{dx\otimes dx}{x^2}+\widehat g_x.
\]
By \cite[(69)--(70)]{BH14}, the level hypersurfaces of \(x\) are
complete, their second fundamental forms and the required covariant
derivatives are uniformly bounded, and their principal curvatures
have strictly positive limits.  Thus the geometric hypotheses of the
Carleman estimate \cite[(71)]{BH14} are satisfied.

\subsection{The Carleman argument}
We use the nonnegative Laplacian
$
 P:=\Delta_{
 \widehat g
 }.
$ Then $\Delta_{g_\Omega}f=0$ is equivalent to $Pf=0$.

By \cite[(71), Section~8]{BH14}, there exist
\(x_0>0\), \(\lambda_0>0\), and \(C>0\) such that
\begin{equation}\label{eq:Carleman}
\begin{split}
 \int_{0<x<x_0}|P\psi|^2x^{-\lambda}\,dV_{\widehat g}
 \geq C\int_{0<x<x_0}
 \bigl(
   \lambda^3|\psi|^2
   +\lambda|\nabla^{\widehat g}\psi|_{\widehat g}^2
   +\lambda^{-1}|(\nabla^{\widehat g})^2\psi|_{\widehat g}^2
 \bigr)
 x^{-\lambda}\,dV_{\widehat g}
\end{split}
\end{equation}
for every \(\lambda\geq\lambda_0\) and every admissible \(\psi\).
As explained immediately after \cite[(71)]{BH14}, the slices are
noncompact but complete, and the required \(O(v)\) decay of \(\psi\)
and its derivatives justifies the integrations by parts at their
noncompact ends.

We next verify this decay for \(f\).  Since any two smooth defining
functions are equivalent on the boundary patch, the hypothesis
\(f=O(\rho^\infty)\) implies
\[
 f=O(r_0^N)
 \qquad\text{for every }N.
\]
The ACH frame derivatives preserve arbitrary-order vanishing, and
the bounded connection coefficients give
\[
 \left|
 (\nabla^{\widehat g_{\rm ch}})^k f
 \right|_{\widehat g_{\rm ch}}
 \leq C_{k,N}r_0^N
 \qquad\text{for all }k,N.
\]
Equations \eqref{eq:Q-estimate} and
\eqref{eq:connection-comparison} give the corresponding estimate for
\(\widehat g\):
\begin{equation}\label{eq:ACH-flatness}
 \left|
 (\nabla^{\widehat g})^k f
 \right|_{\widehat g}
 \leq C_{k,N}r_0^N
 \qquad\text{for all }k,N.
\end{equation}
Consequently,
\[
 f\in C^{k,\alpha}_N
 \qquad\text{for every }k,N,
\]
and Lemma~4.1 of \cite{BH14} yields
\begin{equation}\label{eq:double-flatness}
 f\in C^{k,\alpha}_{N,N}
 \qquad\text{for every }k,N.
\end{equation}

For each fixed \(\lambda\), choose \(N\) sufficiently large.
The arbitrary order double-weighted decay in
\eqref{eq:double-flatness} then dominates both the factor
\(x^{-\lambda}\) and the volume growth at the noncompact ends of the
slices.  Hence
\begin{equation}\label{eq:weighted-L2}
 x^{-\lambda/2}
 (\nabla^{\widehat g})^j f
 \in L^2(dV_{\widehat g}),
 \qquad
 0\leq j\leq2.
\end{equation}

Choose $ 0<a_1<a<b<x_0$
and \(\chi=\chi(x)\in C^\infty\) such that
\[
 \chi=1\quad\text{for }x\leq a,
 \qquad
 \chi=0\quad\text{for }x\geq b.
\]
Let \(\eta\in C^\infty([0,\infty))\) satisfy
\[
 \eta(t)=0\quad(t\leq1),
 \qquad
 \eta(t)=1\quad(t\geq2),
\]
and set
\[
 \eta_\varepsilon(x):=\eta(x/\varepsilon),
 \qquad
 f_\varepsilon:=\eta_\varepsilon\chi f.
\]
The function \(f_\varepsilon\) is supported away from \(x=0\) and
\(x=x_0\), while \eqref{eq:double-flatness} gives the decay required
at the noncompact ends.  Thus \(f_\varepsilon\) belongs to the
admissible class in \eqref{eq:Carleman}.

On the other hand, we have
\begin{equation}\label{eq:commutator}
 [P,\zeta]h
 =
 (P\zeta)h
 -2\langle d\zeta,dh\rangle_{\widehat g}.
\end{equation}
The derivatives of \(\eta_\varepsilon\) are supported in
$
 \{\varepsilon\leq x\leq2\varepsilon\}.
$
Since
$
 \left|\frac{dx}{x}\right|_{\widehat g}=1,
$
the first two normalized derivatives of
\(\eta(x/\varepsilon)\) are bounded independently of
\(\varepsilon\).  It follows from \eqref{eq:weighted-L2} that, for
each fixed \(\lambda\),
\begin{equation}\label{eq:cutoff-error}
 \left\|
 x^{-\lambda/2}
 [P,\eta_\varepsilon](\chi f)
 \right\|_{L^2(dV_{\widehat g})}
 \longrightarrow0
 \qquad(\varepsilon\to0).
\end{equation}
The same argument shows that
$
 f_\varepsilon\longrightarrow\chi f
$
in the weighted Sobolev norm occurring in
\eqref{eq:Carleman}.  We may therefore let
\(\varepsilon\to0\), with \(\lambda\) fixed, and obtain
\eqref{eq:Carleman} for \(\psi=\chi f\).

Since \(Pf=0\), then
$P(\chi f)=[P,\chi]f$ 
and the right-hand side is supported in
\(\{a\leq x\leq b\}\).  Dropping the nonnegative derivative terms in
\eqref{eq:Carleman} gives
\[
 C\lambda^3a_1^{-\lambda}
 \int_{\{x<a_1\}}|f|^2\,dV_{\widehat g}
 \leq
 a^{-\lambda}
 \int_{\{a\leq x\leq b\}}
 |[P,\chi]f|^2\,dV_{\widehat g}.
\]
The last integral is finite by
\eqref{eq:double-flatness} and is independent of \(\lambda\).  Thus
\[
 \int_{\{x<a_1\}}|f|^2\,dV_{\widehat g}
 \leq
 C^{-1}A_\chi\lambda^{-3}
 \left(\frac{a_1}{a}\right)^\lambda
\]
for some \(A_\chi<\infty\).  Letting \(\lambda\to\infty\) yields
\[
 f=0
 \qquad\text{on }\{x<a_1\}\subset E_{p,\varepsilon}.
\]
This proves Theorem~\ref{thm:local-unique-continuation}.

\subsection{Proof of Theorem~\ref{main}}

Let \(\Omega_p^\ast\) denote the connected component of
\(\Omega^\ast\) containing a one-sided neighborhood of \(p\), and set
\(f=\log(J_\Omega/c_n)\). Then \(f\) is real analytic on
\(\Omega_p^\ast\), satisfies \(\Delta_{\widehat g}f=0\), and, by
Theorem~\ref{9-16-lem1}, satisfies \(f=O(\rho^\infty)\) near \(p\).
Theorem~\ref{thm:local-unique-continuation} implies that \(f\) vanishes
on a nonempty open subset of \(\Omega_p^\ast\). The identity principle
for real-analytic functions therefore gives \(f\equiv0\) on
\(\Omega_p^\ast\). Consequently, \(J_\Omega\equiv c_n\) on
\(\Omega_p^\ast\).

Since \(J_\Omega\) is real analytic on
\(\Omega\setminus\{z\in\Omega:K_\Omega(z,z)=0\}\), and this set is
connected, the identity principle again yields
\(J_\Omega\equiv c_n\) on
\(\Omega\setminus\{z\in\Omega:K_\Omega(z,z)=0\}\). It follows that the
Bergman metric is well defined on this set. Hence,
\(\Omega^\ast=\Omega\setminus\{z\in\Omega:K_\Omega(z,z)=0\}\), and the
Bergman metric is Einstein there. This completes the proof of
Theorem~\ref{main}.

\end{document}